\documentclass[10pt, reqno]{amsart}
\usepackage{amsmath}
\usepackage[margin=1.0in]{geometry}
\usepackage[foot]{amsaddr}
\usepackage{amssymb}
\usepackage{physics}
\usepackage{mathrsfs}
\usepackage{bbold}
\usepackage{enumerate}         
\usepackage{color}
\usepackage{cases}
\usepackage{url} 
\usepackage{amsthm}
\usepackage{bm}
\usepackage{xy}
\usepackage[ruled,vlined]{algorithm2e}
\usepackage{mathrsfs}
\usepackage[gen]{eurosym}
\usepackage{graphicx}
\usepackage{comment}
\usepackage{dsfont}

\usepackage[colorinlistoftodos]{todonotes}

\theoremstyle{plain}
\newtheorem{theorem}{Theorem}[section]

\newtheorem{corollary}[theorem]{Corollary}

\newtheorem{lemma}[theorem]{Lemma}

\newtheorem{proposition}[theorem]{Proposition}

\newtheorem{assumption}[theorem]{Assumption}
\theoremstyle{remark}
\newtheorem{remark}[theorem]{Remark}
\newtheorem{example}[theorem]{Example}

\title[A Bayesian Sequential Soft Classification Problem for a B.M.'s Drift]{A Bayesian Sequential Soft Classification Problem for a Brownian Motion's Drift}
\author{Steven Campbell and Yuchong Zhang}
\address{\textit{Corresponding Author:} Steven Campbell, Dept.\ of Statistics, Columbia University}
\email{sc5314@columbia.edu}
\date{\today}

\begin{document}
	
	\maketitle
	\vspace{-0.5cm}
	\begin{abstract}
		In this note we introduce and solve a soft classification version of the famous Bayesian sequential testing problem for a Brownian motion's drift. We establish that the value function is the unique non-trivial solution to a free boundary problem, and that the continuation region is characterized by two boundaries which may \textit{coincide} if the observed signal is not strong enough. By exploiting the solution structure we are able to characterize the functional dependence of the stopping boundaries on the signal-to-noise ratio. We illustrate this relationship and compare our stopping boundaries to those derived in the classical setting.
	\end{abstract}
	
	
	
	\section{Introduction and problem set-up}
		A classic continuous-time problem is the sequential testing of the drift of an arithmetic Brownian motion \cite{gapeev2004wiener,shiryaev1967two,shiryaev2007optimal}. As in that problem, we begin with a stochastic basis $(\Omega,\mathcal{F},\mathbb{F},\mathbb{P})$ supporting a Brownian motion $W$ and a binary random variable $\theta\in\{0,1\}$ independent of $W$. There is an observer who wants to make an inference about $\theta$, but only sees the noisy signal process $X_t=\alpha\theta t+\sigma W_t$,
	where $\alpha\in \mathbb{R}\setminus\{0\}$, $\sigma>0$.
	
	To impose a Bayesian formulation, a prior probability $P(\theta = 1) = \pi$ for $\pi\in(0,1)$ is enforced. The observer may take stock of the information in $X$, until some self-determined $\mathbb{F}^X$-stopping time\footnote{We define $\mathbb{F}^X=(\mathcal{F}_t^X)_{t\geq0}$ to be the filtration generated by the process $X$.} $\tau$. As $X$ is observed, the belief about $\theta$ is updated. We denote by $\Pi$ the posterior probability process $\Pi_t:=\mathbb{P}(\theta=1|\mathcal{F}^X_t)$. Standard arguments from filtering theory lead to the Markovian dynamics
	\[d\Pi_t=\frac{\alpha}{\sigma}\Pi_t(1-\Pi_t)d\overline{W}_t, \ \ \ \Pi_0=\pi,\]
	for the posterior probability process where $\overline{W}_t:=\sigma^{-1}\left(X_t-\alpha\int_0^t\Pi_sds\right)$
	is an $\mathbb{F}^X$-Brownian motion known as the innovation process (c.f. \cite[ (21.0.9)-(21.0.10)]{peskir2006optimal} where the same process is introduced).
	
	In the classic problem, the goal of the observer is to test sequentially the hypotheses $H_0: \theta=0$ and $H_1: \theta=1$ given the information gleaned from the signal process.  When determining the state of nature $\theta$, the observer's response consists of the time $\tau$ the decision is made, and the decision $d \in \{0, 1\}$ itself. The observer then faces the Bayes risk:
	\begin{equation}\label{eqn:classic.bayes.risk}
		\mathbb{E}_\pi\left[c\tau+a_1\mathds{1}_{\{d=0,\theta=1\}}+a_2\mathds{1}_{\{d=1,\theta=0\}}\right],
	\end{equation}
	and searches for a minimizing pair $(\tau^*,d^*)$. Here $a_1,a_2,c>0$ are constants with $c$ representing the cost of observation, and $a_1$ (resp. $a_2$) representing the cost of making a type 1 (resp. type 2) error. 
	
	This is a hard classification problem since a definite decision about the state of nature must be made at $\tau$ even if the observer is not certain. In this work, we relax this requirement to allow for a soft-classification approach. The modified Bayes risk we consider is,
	\begin{equation}\label{eqn:bayes.risk.soft}\mathbb{E}_\pi \left[c\tau+\mathscr{L}(\theta,\Pi_{\tau})\right],
	\end{equation}
	where $\mathscr{L}:\{0,1\}\times[0,1]\to\mathbb{R}_+$ is a loss function satisfying $\mathscr{L}(v,\pi)=0$ if and only if $v=\pi$. By iterated conditioning, we can remove the unknown variable $\theta$ in the loss and write
	\begin{equation}\label{eqn:iter.cond.bayes.risk}
		\mathbb{E}_\pi \left[\mathscr{L}(\theta,\Pi_{\tau})\right] 
		=\mathbb{E}_\pi \left[ \mathbb{E}_\pi \left[\mathscr{L}(0,\Pi_{\tau})\mathds{1}_{\theta = 0}+ \mathscr{L}(1,\Pi_{\tau})\mathds{1}_{\theta = 1}\big|\mathcal{F}^{X}_{\tau}\right]\right]=\mathbb{E}_\pi[g(\Pi_{\tau})],
	\end{equation}
	where
	\[g(\pi):=\pi\mathscr{L}(1,\pi)+(1-\pi)\mathscr{L}(0,\pi).\]
	It is easy to see that $g(0)=g(1)=0$, i.e.\ certainty about the state of nature incurs no penalty. 
	
	This representation allows us to use a Markovian approach for the optimal stopping problem
	\begin{equation}\label{eqn:problem}
		V_*(\pi):=\inf_{\tau\in\mathcal{T}}\mathbb{E}_\pi \left[c\tau+g(\Pi_{\tau})\right],
	\end{equation}
	where $\mathcal{T}$ is the set of (a.s. finite) $\mathbb{F}^\Pi$-stopping times.
	A similar argument in the classic problem allows one to get the same representation, but with the special penalty function:
	\begin{equation}\label{eqn:classic.penalty}
		g(\pi)=a_1\pi\wedge a_2(1-\pi).
	\end{equation}
	Unlike the soft classification functions that we will study here, this function has a kink which arises at the point $\pi=\frac{a_2}{a_1+a_2}$. We will see in the analysis to follow that this leads to non-trivial differences in the solution structure (see Section \ref{sec:illustration} for an illustration).
	
	We will require that the penalty function $g(\cdot)$ induced by the soft-classification loss satisfies some simple properties. The following assumption, which depends on the operator $\mathcal{A}$,
	\begin{equation}
		(\mathcal{A} f)(\pi) :=\frac{1}{2}\pi^2(1-\pi)^2 f''(\pi), \quad f\in C^2(0,1),
	\end{equation}
	is enforced throughout this paper.\footnote{The operator $\mathcal{A}$ is related to the infinitesimal generator $\mathcal{L}$ of $\Pi$ by $\mathcal{L} = (\alpha^2/\sigma^2) \mathcal{A}$.} 
	\begin{assumption}\label{ass:g.scst}\
		\begin{itemize}
			\itemsep = 0em
			\item[(G1)] $g\in C^2(0,1)$ is concave and satisfies $g(0)=g(1)=0$;
			\item[(G2)] There exists a $\pi_0\in(0,1)$ such that $\mathcal{A}g$ is strictly decreasing on $(0,\pi_0)$ and strictly increasing on $(\pi_0,1)$.
		\end{itemize}
	\end{assumption}
	Condition (G2) is a technical condition that ensures the stopping time solution to \eqref{eqn:problem} is well behaved. Both of these conditions are satisfied for the penalty induced by a host of well-known losses. Two examples are highlighted here.
	
	\begin{example}[Cross-entropy loss]\label{ex:g-cross-entropy.scst}
		Suppose the loss is given by $\mathscr{L}(\theta,\pi)=-a_1\theta\log(\pi)-a_2(1-\theta)\log(1-\pi)$ for $a_1,a_2>0$. This is a potentially asymmetric version of the popular cross-entropy loss. The induced penalty function is given on $(0,1)$ by $g(\pi)=-a_1\pi\log(\pi)-a_2(1-\pi)\log(1-\pi)$,
		and we naturally extend the definition of $g$ to the points $\pi=0$ and $\pi=1$ by taking limits.
	\end{example}
	
	\begin{example}[$L_1$ and $L_2$ losses]\label{ex:g-l1-l2.scst}
		The induced penalty function for the $L_1$ loss $\mathscr{L}(\theta,\pi)=\left|\theta-\pi\right|$ is $g(\pi)=2\pi(1-\pi)$. For the $L_2$ loss $\mathscr{L}(\theta,\pi)=\left(\theta-\pi\right)^2$, we recover a constant multiple of the $L_1$ penalty; $g(\pi)=\pi(1-\pi)$.
	\end{example}
	
		Soft-classification losses have experienced an increase in popularity due to their utility in machine learning. When incorporated into the sequential testing framework, we will find that this class of problems has a nice solution structure. The optimal stopping time $\tau$ will be characterized as the first hitting time to two constant boundaries which depend on the information ratio $\alpha/\sigma$ and form part of the solution to a free boundary problem. The existence and uniqueness of the solution has a geometric interpretation related to Bisztriczky's Theorem \cite{bisztriczky1990separated} which (in the special case of 2 dimensions) says that any two strictly separated convex bodies are supported by exactly two tangents. This result was first proved\footnote{According to \cite{lewis1996common}, Bisztriczky declared that his initial proof was incomplete.} by \cite{cappell1994common} and later in novel ways by \cite{castillo2021common,lewis1996common}, amongst others.
	
	The remainder of this note is organized as follows. Section \ref{sec:lit.rev} provides a brief overview of the related literature. Section \ref{sec:free.bdy.problem} studies a free-boundary problem associated with the soft classification formulation and demonstrates that the solution is \textit{as tractable} as its counterpart in the classic problem. Section \ref{sec:verification} then verifies that the value function can be characterized as the unique solution to this free boundary problem and that the optimal stopping boundaries satisfy a pair of transcendental equations. Through this representation, we are able to observe that the value function can be written in terms of the convex envelope of a related function. In Section \ref{sec:inf.ratio} we study the behavior of the optimal stopping boundaries as the information ratio is varied. Finally, in Section \ref{sec:illustration} we illustrate the solution using several sample loss functions, and compare the results with those obtained in the hard-classification problem. Section \ref{sec:inf.hor.conclusion} concludes and suggests directions for future work.

	\subsection{Related Literature}\label{sec:lit.rev}
	This problem belongs to the field of sequential analysis which covers a broad class of statistical problems where data is collected and analyzed sequentially over time, rather than all at once. The birth of the theory can be credited to Wald's seminal 1945 paper \cite{wald1945sequential} which formally introduced sequential hypothesis testing and the Sequential Probability Ratio Test (SPRT). However, in  this treatise Wald himself attributes the first conceptualization of a sequential test to a work of Dodge and Romig \cite{dodge1929method} which dates back even earlier to 1929. The decades that followed Wald's early work were marked by many important developments for the theory. For instance, Wald and Wolfowitz \cite{wald1950bayes,wald1948optimum} proved the optimality of the SPRT and studied Bayes solutions of sequential testing problems, Kiefer and Weiss \cite{kiefer1957some} studied properties of generalized SPRTs alongside a related sample size minimization problem, and Chernoff \cite{chernoff1959sequential} addressed a procedure for the sequential design of experiments. The theory of sequential decision making for stochastic processes in continuous time relevant to this work was introduced in \cite{dvoretzky1953sequential}.  
	
	The classic investigation into the drift of a signal process has received sustained interest since initial works like \cite{shiryaev1967two}. It is studied in the finite horizon setting in \cite{gapeev2004wiener} and extended to a signal process whose drift has a state dependence in \cite{gapeev2011sequential}.  The paper \cite{ekstrom2015bayesian} treats finite and infinite horizon problems with general (i.e. not necessarily binary) drifts and prior distributions for the unknown state $\theta$. The case of the $L_2$ penalty in \eqref{eqn:bayes.risk.soft}, also for general $\theta$, is studied in \cite{ekstrom2022bayesian} by cleverly exploiting the quadratic structure. A version of this $L_2$ problem that incorporates both control and stopping can be found in \cite{ekstrom2021sequential}. The recent work \cite{xu2023decision} has also studied a related sequential decision making problem with two information regimes. It is worth noting that problems with linear costs of observations like \eqref{eqn:problem} are also considered in \cite{irle2004solving} using probabilistic techniques. Here, we hope to obtain explicit characterizations of the value function and stopping boundaries that mimic those obtained for the classic problem (c.f. \cite[Theorem 21.1]{peskir2006optimal}) and this leads us to pursue a different approach.

	\section{The free boundary problem}\label{sec:free.bdy.problem}
	
	As in the classic problem, we will begin by guessing the form of the solution and then verifying that it is correct. From the form of the value function, we suspect that there should be an $A,B\in(0,1)$ such that
	\[\tau_{A,B}:=\inf\{t\geq0:\Pi_t\not\in(A,B)\}\]
	is an optimal stopping time for our problem. 	Let $K:=\alpha^2/(c\sigma^2)>0$ be a simplifying constant. Standard arguments based on the strong Markov property then lead to the following free-boundary problem for an unknown function $V$ and unknown points $0\leq A\leq B\leq 1$:
	\begin{numcases}{}
		\mathcal{A}V(\pi)=-K^{-1} & $\pi\in(A,B)$ \label{eqn:fb1}\\ 
		V(A)=g(A) & \label{eqn:fb2}\\
		V(B)=g(B) & \label{eqn:fb3}\\
		V'(A)=g'(A) & \label{eqn:fb4}\\
		V'(B)=g'(B) & \label{eqn:fb5}\\
		V<g & $\pi\in(A,B)$ \label{eqn:fb6}\\
		V=g & $\pi\in[0,A]\cup[B,1]\label{eqn:fb7}$.
	\end{numcases}
	
	This can be compared with the related formulation in the hard classification setting in \cite[Chapter VI, Section 21.1]{peskir2006optimal}. In the remainder of this section we will characterize the existence and uniqueness of the solution to the above free-boundary problem.
	
	It can be easily verified that for any $A\in(0,1)$ the unique solution to \eqref{eqn:fb1}, \eqref{eqn:fb2}, and \eqref{eqn:fb4} is given by the function
	\begin{equation}\label{eqn:V.pi.A}
		V(\pi;A)=2K^{-1}\Psi(\pi)+H'(A)\pi+H(A)-AH'(A)
	\end{equation}
	where 
	\begin{equation}\label{eqn:Psi}
		\Psi(\pi):=(1-2\pi)\log\left(\frac{\pi}{1-\pi}\right),
	\end{equation}
	and 
	\begin{equation}\label{eqn:H}
		H(\pi):=g(\pi)-2K^{-1}\Psi(\pi).
	\end{equation}
	Equations \eqref{eqn:fb3} and \eqref{eqn:fb5} in terms of the form of $V(\cdot;A)$ read,
	\begin{equation*}
		2K^{-1}\Psi(B)+H'(A)B+H(A)-AH'(A)=g(B) \ \ \ \text{and} \ \ \ 
		2K^{-1}\Psi'(B)+H'(A)=g'(B).
	\end{equation*}
	Then, by using \eqref{eqn:H} we immediately see that they can be equivalently written as
	\begin{equation}\label{eqn:H.tangent.secant}
		H(B)-H(A)=H'(A)(B-A),
	\end{equation}
	\begin{equation}\label{eqn:Hprime.equality}
		H'(A)=H'(B).
	\end{equation}
	Using these definitions, we state here our main theorem which summarizes up front the analysis to follow.
	
	\begin{theorem}\label{thm:free.bdy}
		The free boundary problem \eqref{eqn:fb1}-\eqref{eqn:fb7} admits the trivial solution $V\equiv g$ with any $A= B\in(0,1)$, and a unique non-trivial solution $V$ in $C^2((0,1)\setminus\{A^*,B^*\})\cap C^1(0,1)$ with boundaries $A^*<B^*$ if and only if $\mathcal{A}g(\pi_0)<-K^{-1}$ for $\pi_0\in(0,1)$ from (G2). Moreover, this solution takes the form
		\begin{equation}\label{eqn:sol.free.bdy} V(\pi)=\begin{cases}
				V(\pi;A^*), & \pi\in (A^*,B^*)\\
				g(\pi), & \pi\in [0,A^*]\cup[B^*,1]
			\end{cases}
		\end{equation}
		where $V(\pi;A^*)$ is given by \eqref{eqn:V.pi.A}, and $A^*$ and $B^*$ are the unique solutions to the equations \eqref{eqn:H.tangent.secant}-\eqref{eqn:Hprime.equality} subject to the constraints $A^*\leq \pi_*$, $B\geq\pi^*$ for $\pi_*<\pi^*$ the unique boundary points of the set
		\begin{equation}\label{eqn:set.U}
			U:=\{\pi\in(0,1): \mathcal{A}g(\pi)<-K^{-1}\}.
		\end{equation}
	\end{theorem}
	
	First, we formally investigate the set $U$ appearing in \eqref{eqn:set.U} and its boundary points. Note that by (G2), $\pi\mapsto\mathcal{A}g(\pi)$ is a univariate, unimodal function. This exactly says that it is quasi-convex and so the set $U$ is convex. By continuity it is open, and when $\mathcal{A}g(\pi_0)<-K^{-1}$ it is necessarily nonempty. We also claim that the boundary points $\pi_*$ and $\pi^*$ must always be in $(0,1)$. The next lemma verifies that $\lim_{\pi\downarrow0}\mathcal{A}g=\lim_{\pi\uparrow 1}\mathcal{A}g=0$ so that, eventually, points near $0$ and $1$ satisfy $\mathcal{A}g>-K^{-1}$.
	
	\begin{lemma}\label{lem:Ag.limits}
		$\lim_{\pi\downarrow0}\mathcal{A}g(\pi)=\lim_{\pi\uparrow 1}\mathcal{A}g(\pi)=0$.
	\end{lemma}
	
	\begin{proof}
		To see this observe by the concavity of $g$ that $\mathcal{A}g\leq 0$. With this we will argue by contraction. We will treat the limit $\pi\uparrow 1$ as the other is similar. Suppose that $\lim_{\pi\uparrow 1}\mathcal{A}g(\pi)\leq -\epsilon$ for some $\epsilon>0$ (the limit exists by monotonicity). By Assumption \ref{ass:g.scst} (G2) we must have $\mathcal{A}g\leq -\epsilon$ all $\pi\in(\pi_0,1)$. Then,
		\[-\frac{2\epsilon}{\pi^2(1-\pi)^2}\geq g''(\pi), \ \ \ \forall\pi\in(\pi_0,1).\]
		Using $\pi_0$ as reference point, we see that for $\pi>\pi_0$ we have:
		\begin{align*}
			g(\pi)=g(\pi_0)+\int_{\pi_0}^\pi g'(\pi_0)+\int_{\pi_0}^ug''(v)dvdu&\leq g(\pi_0)+\int_{\pi_0}^\pi g'(\pi_0)+2\epsilon\int_{\pi_0}^u\frac{-1}{v^2(1-v)^2}dvdu\\
			&=g(\pi_0)+g'(\pi_0)\left( \pi-\pi_0 \right)+2\left( 1-2\pi \right) \epsilon\log \left( {\frac {\pi}{1-\pi}}
			\right)+\epsilon C_1\pi+\epsilon C_0
		\end{align*}
		where $C_1$ and $C_0$ are constants depending only on $\pi_0$. Taking $\pi\uparrow 1$ across the above we obtain the contradiction $0=g(1)\leq -\infty$.
	\end{proof}
	
	We now turn to an investigation of the function $H$. This will help us understand if a solution $(A^*,B^*)$ to \eqref{eqn:H.tangent.secant} and \eqref{eqn:Hprime.equality} exists. The following lemma tells us that $\pi_*$ and $\pi^*$ are critical points for $H'$ and allows us to characterize the function behavior.
	
	\begin{lemma}\label{lem:Hprime}
		$H'(\pi)$ is strictly increasing on $(0,\pi_*)$, strictly decreasing on $(\pi_*,\pi^*)$ and strictly increasing on $(\pi^*,\infty)$. Moreover, \[\lim_{\pi\downarrow0} H'(\pi)=-\infty, \ \ \mathrm{and} \ \ \lim_{\pi\uparrow 1}H'(\pi)=\infty.\]
	\end{lemma}
	
	\begin{proof}
		On $(\pi_*,\pi^*)$ we have that $\mathcal{A}g<-K^{-1}=2K^{-1}\mathcal{A}\Psi$ which implies that $g''<2K^{-1}\Psi''$. Then
		\[H''(\pi)=g''(\pi)-2K^{-1}\Psi''<0\]
		and so $H'(\pi)$ is strictly decreasing on this interval. On the other hand, on $(\pi_*,\pi^*)^C$ we have $\mathcal{A}g\geq-K^{-1}=2K^{-1}\mathcal{A}\Psi$ which implies that $g''\geq 2K^{-1}\Psi''$. So,
		\[H''(\pi)=g''(\pi)-2K^{-1}\Psi''\geq 0\]
		and $H'(\pi)$ is increasing. Assumption (G2) allows us to argue that the increase is strict. Lastly, we show $H'(0+)=-\infty$ and $H'(1-)=\infty$. We have
		\[H'(\pi)=g'(\pi)-2K^{-1}\Psi'(\pi)=2K^{-1}\Psi'(\pi)\left[\frac{g'(\pi)}{2K^{-1}\Psi'(\pi)}-1\right].\]
		It is easy to see that $\lim_{\pi\downarrow0}2K^{-1}\Psi'(\pi)=\infty$ and $\lim_{\pi\uparrow 1}2K^{-1}\Psi'(\pi)=-\infty$. Thus, to get our desired result, it suffices to show that 
		\[0\leq \liminf_{\pi\downarrow0}\frac{g'(\pi)}{2K^{-1}\Psi'(\pi)}\leq \limsup_{\pi\downarrow0}\frac{g'(\pi)}{2K^{-1}\Psi'(\pi)}\leq 1-\delta\]
		for some $\delta>0$. We treat the case of small $\pi$, as $\pi$ near $1$ is similar. Recall that $g$ is concave with $g(0)=g(1)=0$, so $g'$ is decreasing and $g'(\pi)\geq0$ for sufficiently small $\pi$. At the same time, $\Psi'(\pi)>0$ for any $\pi<1/2$. Hence the inequality
		$0\leq \frac{g'(\pi)}{2K^{-1}\Psi'(\pi)}$ holds for all sufficiently small $\pi$. Suppose $\lim_{\pi\downarrow0}g'(\pi)<\infty$ (the limit exists by monotonicity). Then, the remaining bound is clear since the denominator diverges to $\infty$. Otherwise $\lim_{\pi\downarrow0}g'(\pi)=\infty$ and we must carefully analyze the limit. By the general form of L'H\^opital's rule:
		\[\limsup_{\pi\downarrow0}\frac{g''(\pi)}{2K^{-1}\Psi''(\pi)}\geq \limsup_{\pi\downarrow0}\frac{g'(\pi)}{2K^{-1}\Psi'(\pi)}\geq\liminf_{\pi\downarrow0}\frac{g'(\pi)}{2K^{-1}\Psi'(\pi)}\geq0\]
		where the last inequality follows from the lower bound above. Observe that for any $0<\pi<1$
		\begin{align*}
			\frac{g''(\pi)}{2K^{-1}\Psi''(\pi)}=\frac{\frac{1}{2}\pi^2(1-\pi)^2g''(\pi)}{\frac{1}{2}\pi^2(1-\pi)^22K^{-1}\Psi''(\pi)}=\frac{\mathcal{A}g}{2K^{-1}\mathcal{A}\Psi}=-\frac{\mathcal{A}g}{K^{-1}}.
		\end{align*}
		Now the limit $\lim_{\pi\downarrow0}\mathcal{A}g$ exists and is equal to $0$ by Lemma \ref{lem:Ag.limits}. At the same time on $(0,\pi_*)$, $\mathcal{A}g\geq -K^{-1}$. 
		In particular, for all sufficiently small $\pi$ there is an $\epsilon>0$ such that $-\mathcal{A}g\leq K^{-1}-\epsilon$. Plugging this bound into the above we obtain:
		\[\frac{g''(\pi)}{2K^{-1}\Psi''(\pi)}\leq \frac{K^{-1}-\epsilon}{K^{-1}}\]
		for all sufficiently small $\pi$. For $\delta=\frac{\epsilon}{K^{-1}}$ we obtain the desired result and conclude $H'(0+)=-\infty$.
	\end{proof}
	
	We also note for future reference that Lemma \ref{lem:Hprime} immediately tells us about the regions of concavity of $H$. This observation will be critical to a subsequent part of the proof. 
	
	\begin{corollary}\label{cor:convexity.concavity.H}
		$H$ is strictly convex on $(0,\pi_*)$, strictly concave on $(\pi_*,\pi^*)$, and strictly convex on $(\pi^*,1)$.
	\end{corollary}
	
	The main idea for finding a solution to \eqref{eqn:H.tangent.secant}-\eqref{eqn:Hprime.equality} is as follows. Given the critical behavior in Lemma \ref{lem:Hprime}, we observe that the unique local maximum and minimum of $H'$ are attained at $\pi_*$ and $\pi^*$, respectively. Using the (local) strict monotonicity of $H$ with the limits $\lim_{\pi\downarrow0} H'(\pi)=-\infty$ and $\lim_{\pi\uparrow 1}H'(\pi)=\infty$ we have that there are two \textit{unique} points $\underline{\pi}\in (0,\pi_*)$ and  $\overline{\pi}\in(\pi^*,1)$ such that $H(\underline{\pi}) = H(\pi^*)$ and $H(\overline{\pi}) = H(\pi_*)$. Using the definition of these 4 critical points we will argue that there is a unique solution $(A^*,B^*)$ to \eqref{eqn:H.tangent.secant}-\eqref{eqn:Hprime.equality} with $A^*\in[\underline{\pi},\pi_*]$ and $B^*\in[\overline{\pi},\pi^*]$. An illustration of these points is given in Figure \ref{fig:pf.illustration}. The crux of this argument will rest on a geometric characterization of the equations \eqref{eqn:H.tangent.secant}-\eqref{eqn:Hprime.equality} which we describe here.
	
	Equation \eqref{eqn:H.tangent.secant} says we need points $(A,B)$ so that the secant connecting $H(A)$ to $H(B)$ has the same slope as the tangent line at $A$. Then, \eqref{eqn:Hprime.equality} says that this slope must also be the slope of the tangent line at $B$. We readily conclude that a pair of points $A\in[\underline{\pi},\pi_*]$ and $B\in[\pi^*,\overline{\pi}]$ solve \eqref{eqn:H.tangent.secant} and \eqref{eqn:Hprime.equality} if and only if they admit a common tangent line for $H$. 
	
	
	Recalling Corollary \ref{cor:convexity.concavity.H} we have that $H$ is strictly convex on $[\underline{\pi},\pi_*]$ and $[\pi^*,\overline{\pi}]$. Hence, to obtain the existence and uniqueness of the points $A^*\in [\underline{\pi},\pi_*]$ and $B^*\in[\pi^*,\overline{\pi}]$ satisfying \eqref{eqn:H.tangent.secant} and \eqref{eqn:Hprime.equality} it suffices to establish the existence and uniqueness of a common tangent line to two strictly convex functions defined on \textit{disjoint} and \textit{strictly separated} domains. To be explicit, define here the functions
	\begin{equation}\label{eqn:f0.f1}
		f_{0}:=H|_{\pi\in[\underline{\pi},\pi_*]}, \ \mathrm{and} \ f_{1}:=H|_{\pi\in[\pi^*,\overline{\pi}]}
	\end{equation}
	formed by restricting $H$ to these sub-domains. For this purpose, we prove the following general lemma using a fundamental geometric result for convex sets known as Bisztriczky's Theorem.

	\begin{figure}[h]
		\begin{center}
			\includegraphics[width=0.35\textwidth]{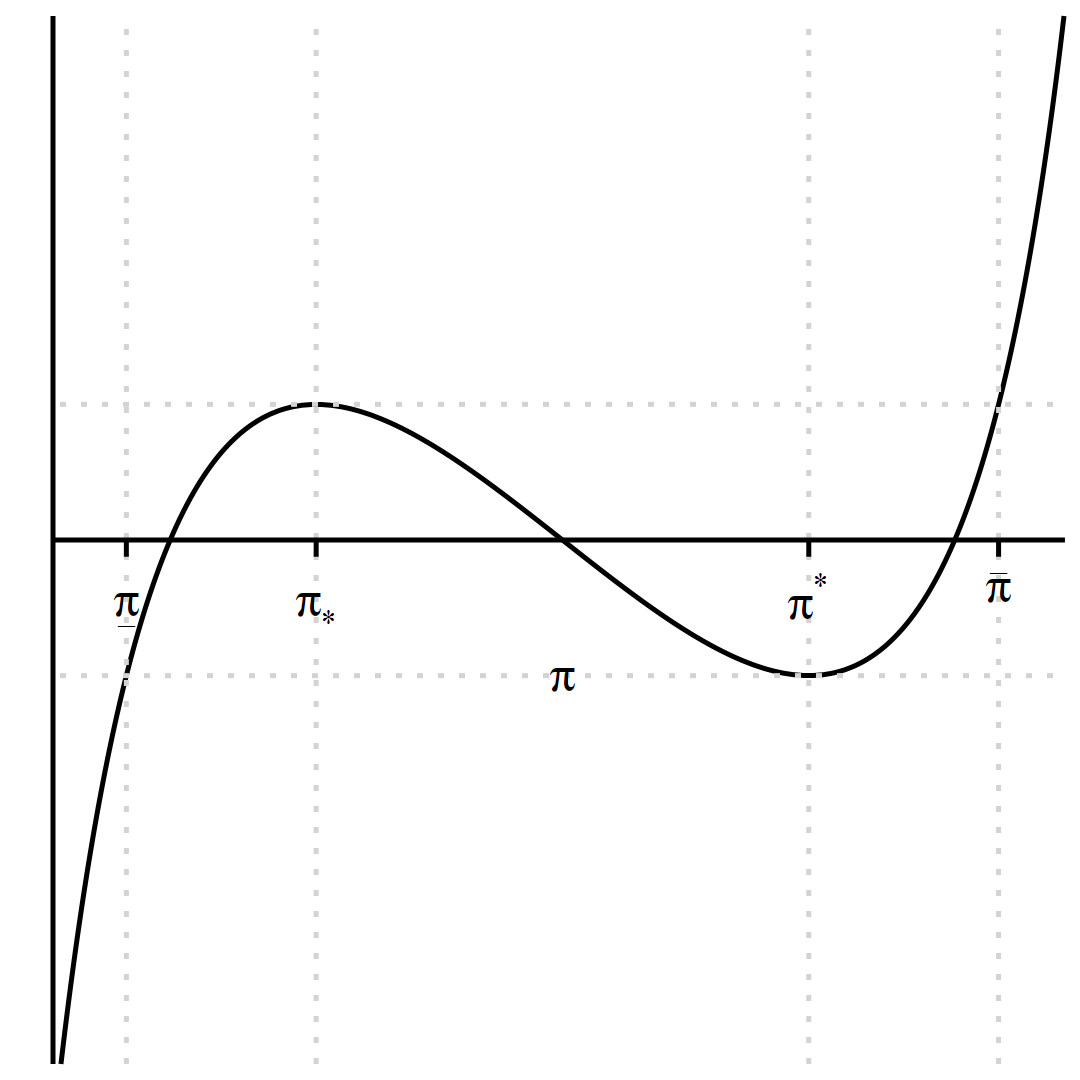}
			\hspace{1cm}
			\includegraphics[width=0.35\textwidth]{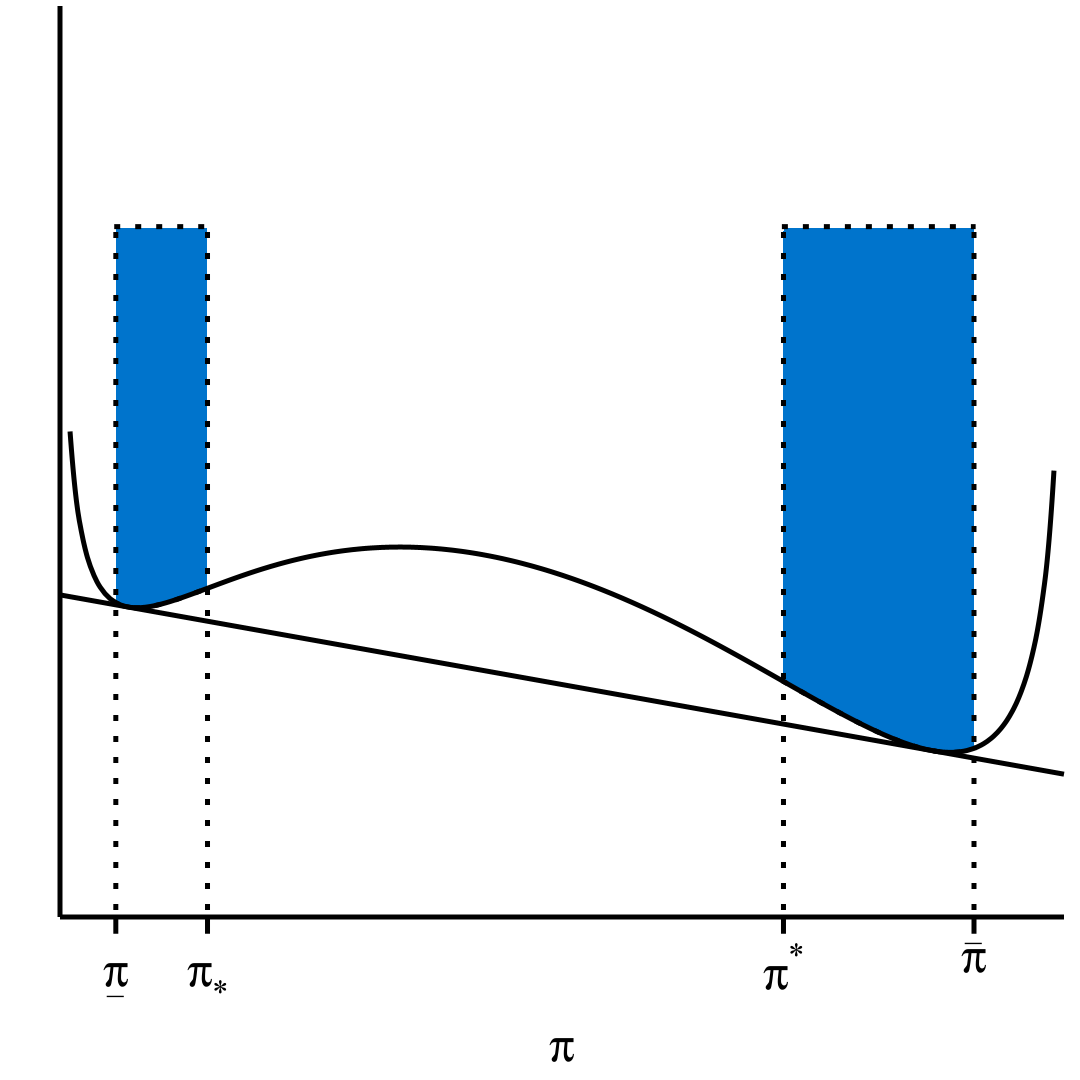}
		\end{center}
		\caption{Visualization of the critical points of $H'$ for the cross-entropy loss (left). Illustration of the arguments of Lemma \ref{lem:common.tangents.convex.func} using an example of $H$ which defines $f_0$, $f_1$ (right). Visualized below the graph of $H$ is the unique common tangent.}
		\label{fig:pf.illustration}
	\end{figure}
	
	\begin{lemma}\label{lem:common.tangents.convex.func}
		Let $X:=[x_0,x_1],Y:=[y_0,y_1]\subset\mathbb{R}$ be strictly separated closed intervals (i.e. $\mathrm{dist}(X, Y)>0$). If $f_X:X\to\mathbb{R}$ and $f_Y:Y\to \mathbb{R}$ are convex and continuous on their domains, then $f_X$ and $f_Y$ have exactly one common tangent line. 
	\end{lemma}
	
	\begin{proof}
		Define $f^*:=\max\{f_X(x_0),f_X(x_1),f_Y(y_0),f_Y(y_1)\}+1$ (Note: $f_X$ and $f_Y$ are bounded by continuity on the compact sets $X$ and $Y$).
		Consider the convex sets:
		$\mathscr{F}_X:=\mathrm{conv}\left\{\mathrm{Gr}(f_X), \{(x_0,f^*)\}, \{(x_1,f^*)\}\right\}$,
		and
		$\mathscr{F}_Y:=\mathrm{conv}\left\{\mathrm{Gr}(f_Y), \{(y_0,f^*)\}, \{(y_1,f^*)\}\right\}$,
		where $\mathrm{Gr}(f_\cdot)$ denotes the graph of $f_\cdot$ and $\mathrm{conv}\{\cdot\}$ denotes the convex hull.
		These sets correspond to the intersection of the epigraph of $f_X$ (resp. $f_Y$) with the half-infinite strip in $X$ (resp. $Y$) ending at the horizontal line $y=f^*$ (see Figure \ref{fig:pf.illustration} for an illustration).
		
		It is clear that $\mathscr{F}_X$ and $\mathscr{F}_Y$ are convex and compact sets with nonempty interior that are strictly separated (if e.g. $f_X$  is constant, the set $\mathscr{F}_X$ is just a rectangle with a vertical height of $1$). Then, the $2$-dimensional case of Bisztriczky's Theorem (c.f. \cite{lewis1996common} or \cite[Corollary 3.12]{castillo2021common}) says that there are exactly two common tangent lines to $\mathscr{F}_X$ and $\mathscr{F}_Y$ with both sets on the same side of the tangent line. It is clear from construction that one such tangent is the line $y=f^*$. We claim the other line must touch the graph of $f_X$ and $f_Y$. This is clear since the vertical lines passing through $(x_0,f^*)$, $(x_1,f^*)$, $(y_0,f^*)$, and $(y_1,f^*)$ are tangent to the left and rightmost faces of $\mathscr{F}_X$ or $\mathscr{F}_Y$ and touch (by construction) only one of the sets. It is also easy to see that the remaining family of tangent lines touching one of the vertices $(x_0,f^*)$, $(x_1,f^*)$, $(y_0,f^*)$, and $(y_1,f^*)$ also cannot (again by construction since the upper faces are level) both be common and keep the sets on the same side. Hence, the only feasible points that the remaining tangent can touch belong to the graphs of $f_X$ and $f_Y$. This completes the proof.
	\end{proof}
	
	With this, we collect our results on $H$ to arrive at the necessary conclusion for the boundaries.
	
	\begin{proposition}\label{prop:Astar.Bstar}
		There exists a unique solution pair $(A^*,B^*)\in[\underline{\pi},\pi_*]\times [\pi^*,\overline{\pi}]$ solving \eqref{eqn:H.tangent.secant} and \eqref{eqn:Hprime.equality}.
	\end{proposition}

	\begin{proof}
		
		The behavior of $H'(\pi)$ given in Lemma \ref{lem:Hprime} tells us that there exists a family of pairs $(A,B)$ such that $H'(A)=H'(B)$. Indeed, we see from the intermediate value theorem that every $\pi\in [\underline{\pi},\pi_*]$ has a unique counterpart $\hat{\pi}\in[\pi^*,\overline{\pi}]$ such that $H'(\pi)=H'(\hat{\pi})$. Using this family of feasible solutions to \eqref{eqn:Hprime.equality}, we now turn to the remaining equation \eqref{eqn:H.tangent.secant}. To determine that there is a unique feasible solution to \eqref{eqn:H.tangent.secant}-\eqref{eqn:Hprime.equality} we need to ensure that \eqref{eqn:H.tangent.secant} is uniquely satisfied by some $A\in[\underline{\pi},\pi_*]$ and $B\in[\pi^*,\overline{\pi}]$ solving \eqref{eqn:Hprime.equality}. By the geometric characterization preceding Lemma \ref{lem:common.tangents.convex.func} we see that it is both necessary and sufficient to show that a unique common tangent exists for the functions $f_0$ and $f_1$ of \eqref{eqn:f0.f1}. This follows immediately from an application of Lemma \ref{lem:common.tangents.convex.func}. Moreover, it is straightforward to check that the tangency points $A^*$ and $B^*$ must be interior to the domains $[\underline{\pi},\pi_*]$ and $[\pi^*,\overline{\pi}]$. This completes the proof.
	\end{proof}

	\begin{remark}[Symmetric Penalty Case]\label{rem:symmetric}
		If $g$ is symmetric about $1/2$ then we want a symmetric solution to \eqref{eqn:fb1}-\eqref{eqn:fb7} and the complexity is substantially reduced. To see this, we use symmetry to search for a solution where $B=1-A$. First, by symmetry, we must have $H'(1/2)=0$ and so by Lemma \ref{lem:Hprime} we obtain $H'(\pi_*)>0$ is a positive local maximum, $H'(\pi^*)<0$ is a negative local minimum and $H'(\pi)$ crosses the $x$-axis at $\pi_0=1/2\in[\pi_*,\pi^*]$. Since $H'(0+)=-\infty$, by the intermediate value theorem and monotonicity there must exist a unique point $\hat{\pi}\in (0,\pi_*)$ such that $H'(\hat{\pi})=0$. By the symmetry of $H$ we must also have $H'(1-\hat{\pi})=0$ and moreover, $1-\hat{\pi}$ is the unique point in $(\pi^*,1)$ with $H'(1-\hat{\pi})=0$. Hence, by the critical behavior outlined above, we have that $(\hat{\pi},1-\hat{\pi})$ are the only pairs of the form $(\pi,1-\pi)$ with $\pi\not=1/2$ satisfying $H'(\pi)=H'(1-\pi)$ as required by \eqref{eqn:Hprime.equality}. This suggests that we can take $A=\hat{\pi}$ and $B=1-\hat{\pi}$. Indeed, using $H'(\hat{\pi})=H'(1-\hat{\pi})=0$ with the symmetry of $H$, it is easy to see that \eqref{eqn:H.tangent.secant} is also satisfied. Thus, this choice of $A$ and $B$ satisfies both consistency equations. As will be shown below in the general case, it is then straightforward to show that the remaining equations for the free boundary problem hold for $V(\pi;A)$.
	\end{remark}

	We are now ready to combine all of these findings to prove our main theorem. 
%
%
%

	\begin{proof}[Proof of Theorem \ref{thm:free.bdy}]
		The first claim in the theorem statement is immediate. On the other hand, the proof of the second claim will proceed in stages.
		
		To obtain sufficiency, suppose $\mathcal{A}g(\pi_0)\geq -K^{-1}$. Then by Assumption \ref{ass:g.scst}(G2) $\mathcal{A}g(\pi)\geq -K^{-1}$ for all $\pi\in(0,1)$. For any fixed $A\in(0,1)$ and all $\pi\in(0,1)$, we obtain $\mathcal{A}g(\pi)\geq -K^{-1}= \mathcal{A}V(\pi;A)$. Now since $\pi^2(1-\pi)^2>0$ for $\pi\in(0,1)$, $g''(\pi)\geq V''(\pi;A)$. As a result, for $\pi>A$
		\begin{align*}
			V(\pi;A)=V(A;A)+\int_A^\pi V'(u;A)du&=g(A)+\int_A^\pi V'(A;A)+\int_A^uV''(v;A)dvdu\\
			&\leq g(A)+\int_A^\pi g'(A)+\int_A^ug''(v)dvdu=g(\pi).
		\end{align*}
		Under (G2) there must be a set of positive measure on which $V''(\cdot;A)< g''$, so the inequality above is strict and there is no possible second intersection point $B$ where \eqref{eqn:fb3} can hold.
		
		Establishing necessity requires a careful analysis that leverages the understanding we developed for the function $H$. For any fixed $A$, we recall that the unique solution to \eqref{eqn:fb1}, \eqref{eqn:fb2}, and \eqref{eqn:fb4} is given by $V(\pi;A)$ in \eqref{eqn:V.pi.A}. By Proposition \eqref{prop:Astar.Bstar} we have that there exists a unique pair $(A^*,B^*)\in[\underline{\pi},\pi_*]\times [\pi^*,\overline{\pi}]$ solving \eqref{eqn:H.tangent.secant} and \eqref{eqn:Hprime.equality} which are equivalent to 
		\eqref{eqn:fb3} and \eqref{eqn:fb5}.
		
%
		
		We now show that it is necessary to have $(A,B)\in[\underline{\pi},\pi_*]\times [\pi^*,\overline{\pi}]$ for the free boundary equations to hold. By the critical behavior in Lemma \ref{lem:Hprime}, if either $A<\underline{\pi}$ or $B>\overline{\pi}$ it clear that there does not exist any pair $(A,B)$ of points where the values of $H'$ coincide as required by \eqref{eqn:Hprime.equality}. We further claim that feasible points $A<B$ for our problem solving \eqref{eqn:Hprime.equality} must be such that $A,B\not\in [\pi_*,\pi^*]$. Once again by Lemma \ref{lem:Hprime} it is clear that we cannot have $A=\pi^*$. On the other hand, if it held that $A\in [\pi_*,\pi^*)$, then
		$V'(\pi;A)-g'(\pi)=H'(A)-H'(\pi)$
		is positive on $(A,\pi^*]$. Consequently,
		$V(\pi;A)-g(\pi)=\int_{A}^\pi H'(A)-H'(u)du>0$
		for $\pi$ near $A$ which violates \eqref{eqn:fb6}. A similar argument holds if $B\in[\pi_*,\pi^*]$. With this, by comparing all the possible arrangements of $A<B$, we conclude that we can restrict our attention to $A\in[\underline{\pi},\pi_*)$ and $B\in(\pi^*,\overline{\pi}]$. 
		
		As there is no possible solution for \eqref{eqn:Hprime.equality} if either $A<\underline{\pi}$ or $B>\overline{\pi}$ we conclude on the strength of Proposition \eqref{prop:Astar.Bstar} that there is a unique pair of points $(A^*,B^*)$ subject to $A^*\leq\pi_*<\pi^*\leq B^*$ satisfying \eqref{eqn:fb3} and \eqref{eqn:fb5}. This implies that the function $V(\pi)$ defined by \eqref{eqn:sol.free.bdy} uniquely satisfies \eqref{eqn:fb1}-\eqref{eqn:fb5} and \eqref{eqn:fb7}.
		
		Therefore, the last task required to complete the proof is to verify \eqref{eqn:fb6} for $V(\pi)$ in \eqref{eqn:sol.free.bdy} which amounts to checking that $V(\pi;A^*)<g(\pi)$ on $(A^*,B^*)$. 
		We have that, 
		\begin{align*}
			V(\pi;A^*)-g(\pi)
			&=\int_{A^*}^\pi V'(u;A^*)-g'(u)du=\int_{A^*}^\pi 2K^{-1}\Psi'(u)+H'(A^*)-g'(u)du=\int_{A}^\pi H'(A^*)-H'(u)du.
		\end{align*}
		Now, $F(\pi):=H'(A^*)-H'(\pi)$ is such that $F(A^*)=F(B^*)=0$, and $F(\pi)$ is decreasing on $(A^*,\pi_*)$, increasing on $(\pi_*,\pi^*)$, and decreasing on $(\pi^*,B^*)$. Consequently, $F(\pi)<0$ near $A^*$, $F(\pi)>0$ near $B^*$ and there exists a unique $\tilde{\pi}\in (A^*,B^*)$ such that $F(\tilde{\pi})=0$. On $(A^*,\tilde{\pi}]$, $V(\pi;A^*)-g(\pi)=\int_{A^*}^\pi F(u)du<0$. Similarly, on $[\tilde{\pi}, B^*)$ we have
		$V(\pi;A^*)-g(\pi)=-\int_{\pi}^{B^*} F(u)du<0$ as required.
%
	\end{proof}
	
	\section{Verification and characterization of the value function}\label{sec:verification}
	
	We will now verify that the solution to the free boundary problem coincides with the value function of the sequential testing problem. The structure of the solution will also reveal that the value function admits the decomposition
	$V_*=2K^{-1}\Psi+H_*$
	where $H_*$ is the convex envelope of $H$.
	
	
	
	
	
\begin{theorem}[Verification]\label{thm:verification}If $\mathcal{A} g(\pi_0)\geq -K^{-1}$, then $V_*\equiv g$. If $\mathcal{A} g(\pi_0)< -K^{-1}$, the value function, $V_*(\pi)$, coincides with $V(\pi)$ in \eqref{eqn:sol.free.bdy}.
	\end{theorem}
	
	\begin{proof}
		The former claim is proved analogously to \cite[Proposition 3.1]{irle2004solving}. The arguments for the latter proceed as in \cite[Theorem 21.1]{peskir2006optimal}. Up to showing the martingale property of a stochastic integral, all verification arguments are standard. To this end, Remark 2.3 of \cite{campbell2024mfgseqtest} gives the necessary estimate $\left|\pi(1-\pi)g'(\pi)\right|\leq M$ for some $M>0$.
	\end{proof}
	
	By standard optimal stopping theory, we easily obtain the form of the smallest optimal stopping time.
	
	\begin{corollary} The smallest optimal stopping time $\tau_*$ of \eqref{eqn:problem} is given by $\tau_*\equiv 0$ if $\mathcal{A} g(\pi_0)\geq -K^{-1}$ and $\tau_*\equiv\tau_{A^*,B^*}:=\inf\{t\geq0:\Pi_t\not\in(A^*,B^*)\}$, otherwise.
	\end{corollary}
	
	We now turn to check the representation of $V_*$ in terms of the convex envelope of $H$. This structure can be intuited from the construction of the solution to the free boundary problem in Section \ref{sec:free.bdy.problem} where $V_*-2K^{-1}\Psi$ is equal to $H$ away from $(A^*,B^*)$ and is an affine function on $(A^*,B^*)$. A formal probabilistic justification is given in the proof of the following theorem.
	
	\begin{theorem}\label{thm:convex.envelope}
		$V_*-2K^{-1}\Psi$ defines the convex envelope of $H$ on $(0,1)$ and satisfies:
		\[V_*(\pi)-2K^{-1}\Psi(\pi)=\inf_{\tau\in\mathcal{T}}\mathbb{E}_\pi \left[H(\Pi_{\tau})\right].\]
	\end{theorem}
	\begin{proof}
		First observe as above that by our construction in Section \ref{sec:free.bdy.problem}, $V_*-2K^{-1}\Psi$ is a convex function that minorizes $H$. We will say that a stopping time $\tau$ is regular if it is bounded by the first exist time of $\Pi$ from some compact interval. With this definition we are able to obtain the chain of inequalities: 
		\begin{align*}
			V_*(\pi)-2K^{-1}\Psi(\pi)&=\inf_{\tau\in\mathcal{T}}\mathbb{E}_\pi \left[c\tau+g(\Pi_{\tau})-2K^{-1}\Psi(\pi)\right]\\
			&=\inf_{\tau\in\mathcal{T}: \ \tau \ \mathrm{is \ regular}}\mathbb{E}_\pi \left[c\tau+H(\Pi_{\tau})+2K^{-1}\Psi(\Pi_{\tau})-2K^{-1}\Psi(\pi)\right]\\
			&=\inf_{\tau\in\mathcal{T}: \ \tau \ \mathrm{is \ regular}}\mathbb{E}_\pi \left[\int_0^\tau c+2K^{-1}\mathcal{L}\Psi(\Pi_t)dt+H(\Pi_{\tau})\right]\\
			&=\inf_{\tau\in\mathcal{T}:\  \ \tau \ \mathrm{is \ regular}}\mathbb{E}_\pi \left[H(\Pi_{\tau})\right]\\
			&\geq \inf_{\tau\in\mathcal{T}}\mathbb{E}_\pi \left[H(\Pi_{\tau})\right]\\
			&\geq \inf_{\tau\in\mathcal{T}}\mathbb{E}_\pi \left[V_*(\Pi_{\tau})-2K^{-1}\Psi(\Pi_{\tau})\right]\\
			&\geq V_*(\pi)-2K^{-1}\Psi(\pi).
		\end{align*}
		In the second equality we have used \cite[Proposition 2.1]{irle2004solving} to obtain that it is sufficient to take the infimum defining $V_*$ over regular stopping times. The third equality then follows by Dynkin's formula, and the fourth by the identity $2K^{-1}\mathcal{L}\Psi=-c$. The last inequality follows from convexity, Jensen's inequality, and the optional sampling theorem since $\Pi$ is a bounded martingale. We obtain that:
		\[V_*(\pi)-2K^{-1}\Psi(\pi)=\inf_{\tau\in\mathcal{T}}\mathbb{E}_\pi \left[H(\Pi_{\tau})\right].\]
		Now, suppose $F$ is any other convex function lying below $H$. We obtain again:
		\[\inf_{\tau\in\mathcal{T}}\mathbb{E}_\pi \left[H(\Pi_{\tau})\right]\geq \inf_{\tau\in\mathcal{T}}\mathbb{E}_\pi \left[F(\Pi_{\tau})\right]\geq F(\pi),\]
		which gives $V_*(\pi)-2K^{-1}\Psi(\pi)\geq F(\pi)$ and so $V_*-2K^{-1}\Psi$ is the largest convex minorant of $H$. The claim follows.
	\end{proof}
	
	This characterization is reminiscent of the result for problems without a linear time penalty in classic optimal stopping theory  where the value function is the largest subharmonic minorant of the cost function.
	
	\section{Asymptotic behavior in the information ratio}\label{sec:inf.ratio}
	
	If $c=1$ then $K=(\alpha/\sigma)^2$ is the squared information/signal-to-noise ratio. This is the main parameter of the sequential testing problem and it has an intuitive interpretation. In this section we are interested in studying the asymptotic behavior of the solution in $K$ and more generally, how the problem behaves as $K$ is varied. To facilitate this analysis, in the sequel all the constants in Section \ref{sec:free.bdy.problem} and \ref{sec:verification} will have their dependence on $K$ emphasized (i.e. we will write $A^*=A^*(K)$, etc.).  
	
	Define $\beta:=\max_{\pi\in[0,1]}|\mathcal{A}g|$ (which is well defined and equal to $|\mathcal{A}g(\pi_0)|$ by Assumption \ref{ass:g.scst}) so that $\mathcal{A}g\geq -\beta$. It is clear from Theorem \ref{thm:verification} that the boundaries are only well defined if $K>\beta^{-1}$ and so we will restrict our attention to this setting. The first result of this section establishes the differentiability of the boundaries in $K$ and an expression for the derivatives.

	\begin{proposition}
		$A^*(K)$ and $B^*(K)$ are continuously differentiable on $(\beta^{-1},\infty)$ with:
		\[\frac{d}{dK}A^*(K)=\frac{2}{K^2}\frac{\Psi\left(B^*(K)\right) -\Psi \left(A^*(K) \right) - \Psi '\left(A^*(K)\right)\left(B^*(K)-A^*(K)\right)}{\left[g''(A^*(K))-2K^{-1}\Psi''(A^*(K))\right]\left(B^*(K)-A^*(K)\right)},
		\]
		\[\frac{d}{dK}B^*(K)=\frac{2}{K^2}\frac{\Psi\left(B^*(K)\right) -\Psi \left(A^*(K) \right) - \Psi '\left(B^*(K)\right)\left(B^*(K)-A^*(K)\right)}{\left[g''(B^*(K))-2K^{-1}\Psi''(B^*(K))\right]\left(B^*(K)-A^*(K)\right)}.
		\]  
	\end{proposition}
	
	\begin{proof}
		We have that $A^*(K)$ and $B^*(K)$ are the unique solutions on $(\pi_*(K),\pi^*(K))^C$ of equations \eqref{eqn:H.tangent.secant}-\eqref{eqn:Hprime.equality}. Write $H=H(\pi;K)$ to emphasize the dependence on $K$. Then, these equations read:
		\begin{equation}\label{eqn:pf.1}
			\partial_\pi H(A;K)-\partial_\pi H(B;K)=0, \ \ \ 
			H(B;K)-H(A;K)-\partial_\pi H(A;K)(B-A)=0.
		\end{equation}
		We will endeavor to apply the implicit function theorem. Treating the left hand side of the equalities above as functions of $(K,A,B)$ we obtain that the determinant of the Jacobian matrix for the partial derivatives with respect to $A$ and $B$ is:
		\[\partial_{\pi\pi} H(A;K)\left[\partial_{\pi}H(B;K)-\partial_\pi H(A;K)\right]-\left[-\partial_{\pi\pi}H(B;K)\right]\left[-\partial_{\pi\pi} H(A;K)(B-A)\right].\]
		By the satisfaction of equations \eqref{eqn:pf.1} at the tuple $(K,A,B)$ this reduces to:
		\[-\partial_{\pi\pi}H(B;K)\partial_{\pi\pi} H(A;K)(B-A).\]
		Now by Lemma \ref{lem:Hprime} we have that both $\partial_{\pi\pi}H(A;K)\not=0$ and $\partial_{\pi\pi}H(B;K)\not=0$ when the equations are satisfied. Moreover, we must have $B>A$ and so the determinant is non-zero. It follows from the implicit function theorem that there exist continuously differentiable solutions $A^*(K)$ and $B^*(K)$ to the equations in a neighborhood of an admissible $K$ such that (after simplification):
		\begin{align*}\frac{d}{dK}&A^*(K)=\frac{\partial_KH(B^*(K);K)-\partial_KH(A^*(K);K)-\partial_{K}\partial_\pi H(A^*(K);K)(B^*(K)-A^*(K))}{\partial_{\pi\pi}H(A^*(K);K)(B^*(K)-A^*(K))},
		\end{align*}
		\begin{align*}\frac{d}{dK}&B^*(K)=\frac{\partial_KH(B^*(K);K)-\partial_KH(A^*(K);K)-\partial_{K}\partial_\pi H(B^*(K);K)(B^*(K)-A^*(K))}{\partial_{\pi\pi}H(B^*(K);K)(B^*(K)-A^*(K))}.
		\end{align*}
		Evaluating these expressions then gives the equations in the statement of the Proposition.
	\end{proof}
	
	This gives us an understanding of how the boundaries change locally. On the other hand, is also possible to characterize the limiting behavior of the boundaries.
	
	\begin{proposition}\mbox{}
		\begin{enumerate}
			\item[(i)] $A^*(K)$ is decreasing and $B^*(K)$ is increasing.
			\item[(ii)] If $K\uparrow \infty$ then $A^*(K)\downarrow 0$ and $B^*(
			K)\uparrow 1$.
			\item[(iii)] If $K\downarrow \beta^{-1}$ then $A^*(K)\uparrow \pi_0$ and $B^*(K)\downarrow \pi_0$.
		\end{enumerate}
	\end{proposition}
	\begin{proof}
		(i) follows from the fact that the derivative of $A^*(K)$ (resp. $B^*(K)$) must be positive (resp. negative). This can be seen since both denominators in the expressions must be (strictly) positive when evaluated at the boundaries (see Lemma \ref{lem:Hprime}) and the numerator is negative for $A^*(K)$ (resp. positive for $B^*(K)$) by the concavity of $\Psi$. 
		
		For (ii) first observe that $\pi_*(K)\downarrow 0$ (resp. $\pi^*(K)\uparrow 1$) as $K\to\infty$. This can be seen from the definition of $\pi_*(K)$ and $\pi^*(K)$ as the boundary points of \eqref{eqn:set.U}, the strict monotonicity of $\mathcal{A}g$ away from $\pi_0$ (see Assumption \ref{ass:g.scst}), and Lemma \ref{lem:Ag.limits} where it is obtained that $\lim_{\pi\downarrow0}\mathcal{A}g(\pi)=\lim_{\pi\uparrow 1}\mathcal{A}g(\pi)=0$. Since $0\leq A^*(K)\leq \pi_*(K)$ and $1\geq B^*(K)\geq \pi^*(K)$ the claim follows.
		
		To obtain (iii) we reason using the points $\underline{\pi}(K)$ and $\overline{\pi}(K)$ defined after Corollary \ref{cor:convexity.concavity.H} by \begin{equation}\label{eqn:underline.overline.pi}
			H'(\underline{\pi}(K);K)=H'(\pi^*(K);K) \ \ \ \mathrm{and} \ \ \   H'(\overline{\pi}(K);K)=H'(\pi_*(K);K)
		\end{equation} (here the dependence of $H$ on $K$ is again emphasized). Since $\pi_0\geq A^*(K) \geq \underline{\pi}(K)$ and $\pi_0\leq B^*(K)\leq \overline{\pi}(K)$, it suffices to show that  $ \underline{\pi}(K)\to \pi_0$ and $ \overline{\pi}(K)\to \pi_0$ as $K\downarrow \beta^{-1}$. As in (ii) it is not hard to see that $\pi_*(K)\uparrow \pi_0$ and $\pi^*(K)\downarrow \pi_0$ as $K\downarrow \beta^{-1}$. Let $(K_n)_{n\geq0}$ be a sequence such that $K_n\downarrow \beta^{-1}$. For $K_n>\beta^{-1}$, $\underline{\pi}(K_n)\leq \pi_0\leq \overline{\pi}(K_n)$ form the unique solution to \eqref{eqn:underline.overline.pi}. Since they are bounded, the Bolzano–Weierstrass theorem allows us to extract a convergent subsequence: $\underline{\pi}(K_{n'})\to a$ $\ \overline{\pi}(K_{n'})\to b$ as $n'\to\infty$ where $a\in[0,\pi_0]$ and $b\in[\pi_0,1]$. Taking limits along the subsequence across $\eqref{eqn:underline.overline.pi}$ and applying the (joint) continuity of $H'$ in $(\pi,K)$ we obtain $H'(a;\beta^{-1})=H'(\pi_0;\beta^{-1})$ and $H'(b;\beta^{-1})=H'(\pi_0;\beta^{-1})$.
		It is easy to check as in Lemma \ref{lem:Hprime} that $H'(\cdot; \beta^{-1})$ is strictly increasing on $(0,1)$. Hence, we must have $a=b=\pi_0$ and since the sequence was arbitrary we conclude as required that $ \underline{\pi}(K)\to \pi_0$ and $ \overline{\pi}(K)\to \pi_0$ as $K\downarrow \beta^{-1}$.
	\end{proof}
	
	Our understanding of the asymptotic behavior as $K\to \infty$ can be refined. In particular, it is possible to obtain a general bound on the rate at which $A^*(K)$ tends to $0$. 
	
	\begin{proposition}\label{prop:asymp.bds}
		For any $\epsilon>0$ there exists a sufficiently large $K_0=K_0(\epsilon)>\beta^{-1}$ and a constant $C=C(K_0,\epsilon)>0$ such that:
		\[A^*(K)\leq \frac{1}{1+CK^{1-\epsilon}}, \ \  B^*(K)\geq  \frac{CK^{1-\epsilon}}{1+CK^{1-\epsilon}}, \ \ \ \forall K\geq K_0.\]
	\end{proposition}
	
	\begin{proof}
		We will treat $A^*(K)$ as the estimate for $B^*(K)$ is obtained similarly. Manipulating the expression for the derivative we obtain:
		\[\frac{1}{K^2}\frac{\left[\Psi\left(B^*(K)\right) -\Psi \left(A^*(K) \right) - \Psi '\left(A^*(K)\right)\left(B^*(K)-A^*(K)\right)\right]A^*(K)^2(1-A^*(K))^2}{\left[\mathcal{A}g(A^*(K))+K^{-1}\right]\left(B^*(K)-A^*(K)\right)}.\]
		By concavity of $\Psi$ the numerator is negative and for $0<\pi<\pi_*(K)$, 
		\[K^{-1} >\mathcal{A}g(\pi)+K^{-1}> 0.\]
		With this we obtain that:
		\[\frac{1}{K}\frac{\left[\Psi\left(B^*(K)\right) -\Psi \left(A^*(K) \right) - \Psi '\left(A^*(K)\right)\left(B^*(K)-A^*(K)\right)\right]A^*(K)^2(1-A^*(K))^2}{\left(B^*(K)-A^*(K)\right)}\]
		is an upper bound on the derivative. Now $\Psi< 0$ for all $\pi\not=1/2$ so we further obtain the upper bound
		\[\frac{d}{dK}A^*(K)<\frac{1}{K}\left[-\frac{\Psi \left(A^*(K) \right)}{\left(B^*(K)-A^*(K)\right)} - \Psi '\left(A^*(K)\right)\right]A^*(K)^2(1-A^*(K))^2.\]
		Note that for any $\epsilon'>0$ there is a sufficiently large $K'>0$ such that $B^*(K)-A^*(K)>1-\epsilon'$ for all $K>K'$. Hence, for sufficiently large $K$:
		\[\frac{d}{dK}A^*(K)<\frac{1}{K}\left[-\frac{\Psi \left(A^*(K) \right)}{1-\epsilon'} - \Psi '\left(A^*(K)\right)\right]A^*(K)^2(1-A^*(K))^2.\]
		Now as $\pi\downarrow 0$,
		$-\pi(1-\pi)\Psi(\pi)\downarrow 0$ so again, for all $\epsilon>0$ there is a $K''>K'$ such that if $K>K''$
		\[-\frac{A^*(K)(1-A^*(K))\Psi \left(A^*(K) \right)}{1-\epsilon'}<\epsilon.\]
		On the other hand, for all sufficiently small $\pi$ we have that $-\pi(1-\pi)\Psi'(\pi)<-1$. We conclude that for any $\epsilon>0$ and sufficiently large $K$:
		\[\frac{d}{dK}A^*(K)<-\frac{1-\epsilon}{K}A^*(K)(1-A^*(K)).\]
		The claim in the proposition then follows by a comparison principle for ordinary differential equations.
	\end{proof}
	
	It is not difficult to see that when $g$ is symmetric so $B^*(K)=1-A^*(K)$, this symmetry (and the symmetry of $\Psi$) can be exploited in the proof to take $\epsilon=0$ in the growth estimates of the boundaries.
	
	\begin{remark}
		Alternative bounds on the boundaries for a particular $g$ can be recovered by considering the behavior of the points $\pi_*(K)$ and $\pi^*(K)$. This is since $A^*(K)\leq \pi_*(K)$ and $B^*(K)\geq \pi^*(K)$. In the case when $g$ corresponds to the symmetric cross entropy loss (see Example \ref{ex:g-cross-entropy.scst} with $a_1=a_2=1$) we get the explicit equations:
		\[\pi_*(K)=\frac{1}{2}-\frac {1}{2}\sqrt {1-\frac{8}{K}}, \ \ \ \pi^*(K)=\frac{1}{2}+\frac {1}{2}\sqrt {1-\frac{8}{K}}\]
		It is easy to see that asymptotically these bounds behave comparably to the ones in Proposition \ref{prop:asymp.bds} (once symmetry is used to take $\epsilon=0$). On the other hand when $g$ corresponds to the $L_1$ loss (see Example \ref{ex:g-l1-l2.scst}) we get:
		\[\pi_*(K)=\frac{1}{2}-\frac {1}{2}\sqrt {1-\sqrt{\frac{8}{K}}}, \ \ \ \pi^*(K)=\frac{1}{2}+\frac {1}{2}\sqrt {1-\sqrt{\frac{8}{K}}},\]
		and the asympotic bounds in Proposition \ref{prop:asymp.bds} are superior.
	\end{remark}
	
	%
	%
	
	\section{Solution illustration and comparison with the classic problem}\label{sec:illustration}
	
	In this section we illustrate the solution to the soft classification sequential testing game for the penalty functions induced by the (symmetric) cross entropy and $L_1$ losses (see Examples \ref{ex:g-cross-entropy.scst} and \ref{ex:g-l1-l2.scst}, and take $a_1=a_2=1$ for the cross-entropy loss). We also compare the solutions to those arrived at in the (symmetric) classic setting for the same parameters. The classic penalty function is given in \eqref{eqn:classic.penalty} and we similarly take $a_1=a_2=1$ there.
	
	\begin{figure}[h!]
		\begin{center}
			\includegraphics[width=0.35\textwidth]{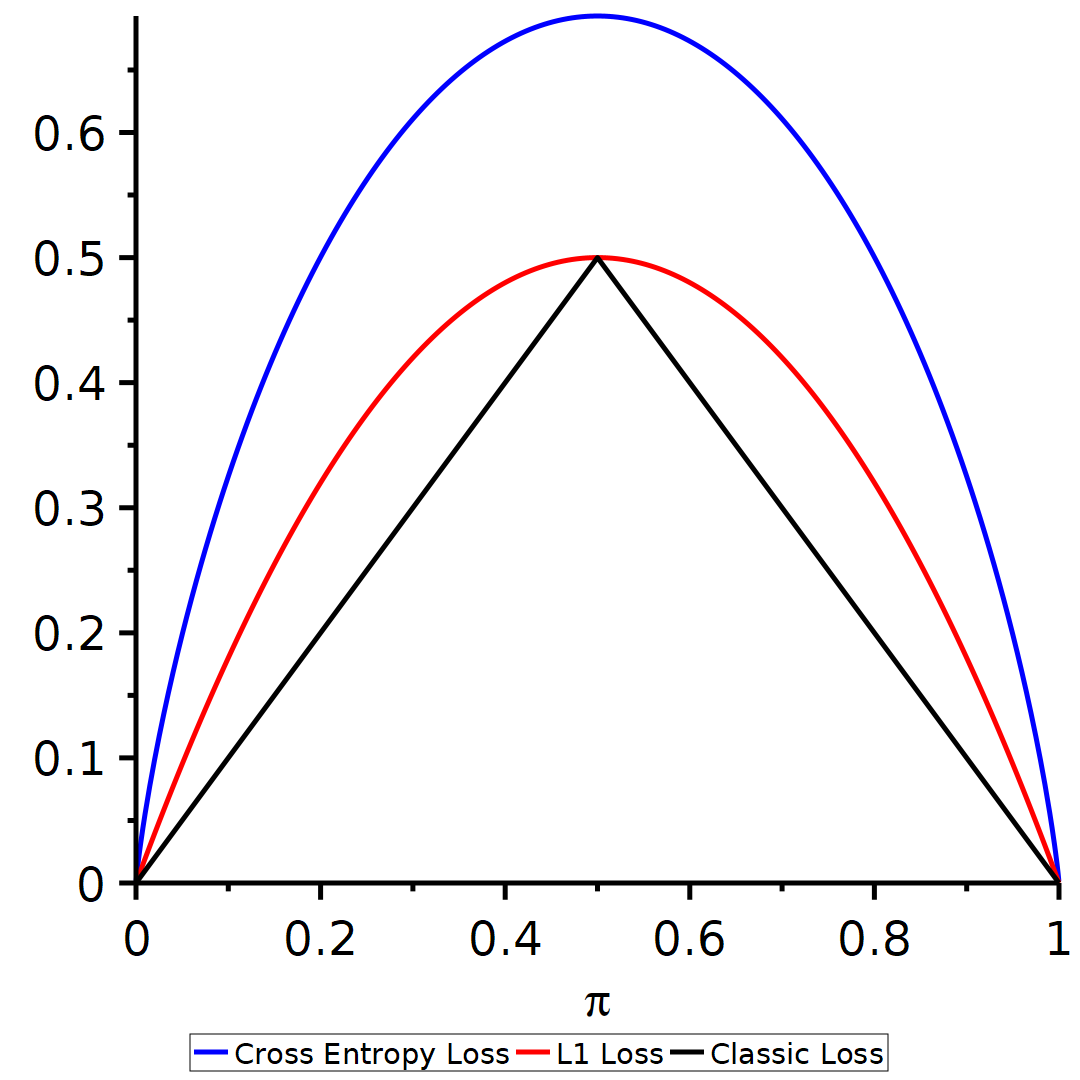}
			\hspace{1cm}
			\includegraphics[width=0.35\textwidth]{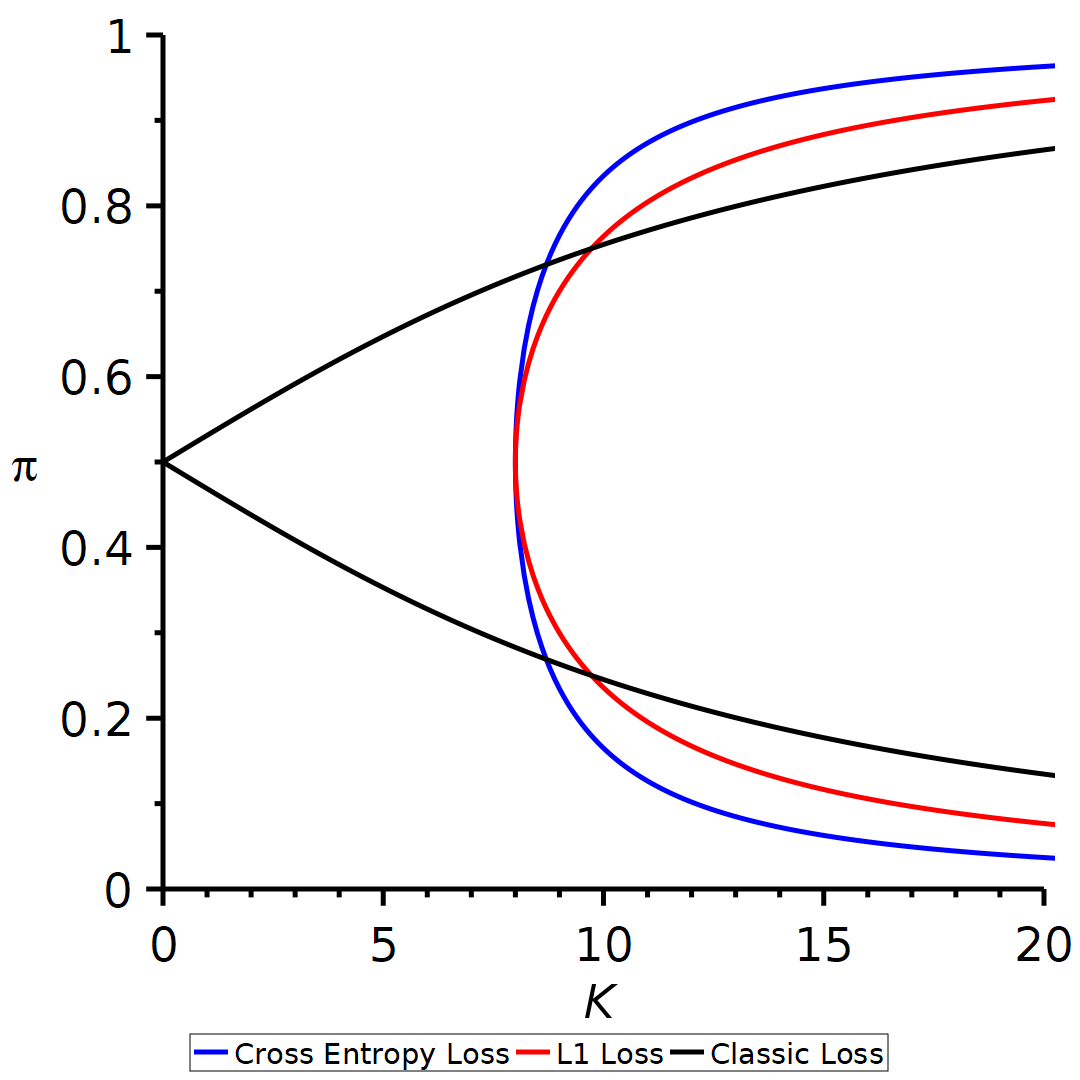}
		\end{center}
		\caption[Penalties $g$ induced by loss functions and the associated optimal stopping boundaries as functions of $K$.]{Penalties $g$ induced by loss functions (left) and the associated optimal stopping boundaries as functions of $K$ (right).}
		\label{fig:solution}
	\end{figure}
	
	Figure \ref{fig:solution} (left) illustrates the three penalty functions induced by each of the losses, and 	Figure \ref{fig:solution} (right) illustrates the corresponding optimal stopping boundaries $A^*$ and $B^*$ as a function of $K$. The upper curves correspond to $B^*$ and the lower curves correspond to $A^*$. The solutions corresponding to each of these losses exhibit qualitatively different behavior. We can observe in Figure \ref{fig:solution} (right) that in accordance with Theorem \ref{thm:verification} the boundaries corresponding to the cross-entropy and $L_1$ losses are well defined only for $K>8$. On the other hand, the classic problem always has a non-empty continuation region and well defined boundaries for any $K>0$.
	
	We can also observe that the boundaries for the $L_1$ loss appear to be contained in the boundaries for the cross-entropy loss. On the other hand, the ordering on the classic boundaries and the soft-classification boundaries switches as $K$ changes. For small $K$ the soft-classification boundaries (both cross-entropy and $L_1$) are contained in the classic boundaries, while this is no longer true for large $K$. Interestingly, if we fix the running cost $c$ this says that for large information ratios it is worth observing the signal process for longer than in the classic problem, while for small information ratios it is often best to make a stopping decision more quickly.
	
	\section{Conclusion}\label{sec:inf.hor.conclusion}
	In this note we have formulated and solved a soft classification version of the famous Bayesian sequential testing problem for a Brownian motion's drift. Our analysis allows us to obtain a semi-explicit solution that mirrors the one arising in the classic setting. By leveraging this solution structure we are also able to obtain a deep understanding of the problem behavior in the information ratio $\alpha/\sigma$. This behavior is then clearly illustrated numerically for several examples of interest. The introduction of these soft classification problems leads to several natural research directions.
	
	\begin{enumerate}
		\item[(i)] Since this problem exhibits interesting and non-trivial behavior in the information ratio, it is potentially worthwhile to investigate how this ratio influences the solutions to other related problems. For instance, the quickest detection problem (c.f. \cite{shiryaev1965some,shiryaev1963optimum}) also depends crucially on this parameter.
		\item[(ii)] When Assumption \ref{ass:g.scst} (G2) is eliminated it is no longer necessary that the continuation region is connected. However, a result akin to Theorem \ref{thm:convex.envelope} should still hold and genuine insights may be generated by studying the types of continuation regions that can arise.
		\item[(iii)] This work has left the finite horizon case $T<\infty$ open.  In view of our results, it is to be expected that the soft classification version of the finite horizon problem also remains as tractable as in the classic formulation.
		\item[(iv)] Since we have observed that the solution is strongly influenced by the information ratio, it is interesting to embed this sequential testing problem into a game setting where players interact and influence the ratio through their actions in pursuit of determining the state $\theta$.
	\end{enumerate}
	Directions (iii) and (iv) are the subject of the companion work \cite{campbell2024mfgseqtest}.

	\section*{Acknowledgment}
	This work was partially supported by an NSERC Alexander Graham Bell Canada Graduate Scholarship (Application No. CGSD3-535625-2019). S. Campbell would like to also thank Ioannis Karatzas, Georgy Gaitsgori and Richard Groenewald for several stimulating discussions during the Optimal Stopping Seminar Series at Columbia University, and for their suggestion to release this short note.
	
	\bibliographystyle{abbrv}
	\bibliography{main.bib}

\end{document}